\documentclass[10pt]{article}

\usepackage{amsmath,latexsym,amssymb}

\usepackage{esint}
\usepackage{graphicx}
\usepackage{bm}
\usepackage{comment}
\usepackage{epsfig}

\usepackage{mathrsfs}
\usepackage{enumerate}
\usepackage{color}
\usepackage{amsfonts}
\usepackage{verbatim}
\usepackage{amsthm}
\usepackage{dsfont}
\usepackage{url}
\usepackage{esint}
\usepackage{tocloft}
\usepackage{tikz}
 \usetikzlibrary{arrows}    
 \usetikzlibrary{calc}  
\usetikzlibrary{decorations.markings}
\usepackage[mathcal]{euscript}

\theoremstyle{definition}
\newtheorem{definition}{Definition}[section]

\theoremstyle{remark}

\theoremstyle{plain}
\newtheorem{theorem}[definition]{Theorem}
\newtheorem{lemma}[definition]{Lemma}

\numberwithin{equation}{section}

\usepackage{authblk}
\title{A representation for the reproducing kernel of a weighted Bergman space}

\author[1]{Erwin Mi\~{n}a-D\'{\i}az \thanks{Corresponding author; Email: minadiaz@olemiss.edu}}
\affil[1]{The University of Mississippi,
Department of Mathematics,
Hume Hall 305,P.O. Box 1848,
University, MS 38677-1848, USA.}

\date{}                     
\newcommand{\cj}{\overline}
\newcommand{\D}{\mathbb{D}}

\newcommand{\N}{\mathbb{N}}
\newcommand{\C}{\mathbb{C}}

\begin{document}
\maketitle
\vspace{-.7cm}\centerline{Anal. Math. Phys. 13, 58 (2023).}
\begin{abstract}
For a weight function in the unit disk which is the modulus of a finite product of powers of Blaschke factors, we give a canonical representation for the  reproducing kernel of the corresponding weighted Bergman space in terms of the values of the kernel and its derivatives at the origin. This yields a formula for the contractive zero divisor of a Bergman space corresponding to a finite zero set.  
\end{abstract}
\section{Statement of results}
Let $\D=\{z\in\C:|z|<1\}$ be the unit disk and let $\sigma$ denote the normalized area measure $d\sigma =\pi^{-1}dA$, so that $\sigma(\D)=1$. Let $h\geq 0$ be an integrable function in $\D$, and let $A^2_h$ denote the space of functions $f$ analytic in $\D$ such that 
\begin{align*}
\int_{\D}|f(\zeta)|^2h(\zeta)d\sigma(\zeta)<\infty.
\end{align*}
We write $A^2$ for $h\equiv 1$. 

If $h$ is such that point evaluation functionals $z\mapsto f(z)$ acting on $A^2_h$ are bounded (for instance, if for some $\rho<1$, $\inf_{\rho<|z|<1}h>0$), then there exists a unique function $K(z,\zeta)$ in $\D\times \D$, analytic in $z$ and in $\cj{\zeta}$, such that 
\begin{align}\label{reproducingproperty}
f(z)=\int_{\D}f(\zeta)K(z,\zeta)h(\zeta)d\sigma(\zeta), \quad f\in A^2_h,\ z\in \D.
\end{align}

The function $K(z,\zeta)$ is called the reproducing kernel of the weighted Bergman space $A^2_h$, and has the symmetry property 
\begin{align}\label{symmetryproperty}
	K(z,\zeta)=\cj{K(\zeta,z)}\,.
\end{align}

In this paper, we give a canonical formula for the reproducing kernel corresponding to a weight $h$ of the form 
\begin{align}\label{generalweight}
	h(z)=\prod_{k=1}^s\left|\frac{z-a_k}{1-\cj{a}_kz}\right|^{p_k}\,,
\end{align}
where $s\geq 1$, $p_k>-2$, and $a_k\in \D$ for all $1\leq k\leq s$,  the points $a_k$  being pairwise distinct. The formula is explicit in terms of the values of $K(z,\zeta)$ and its derivatives at the origin, which can be numerically computed by integration.

The subclass of \eqref{generalweight} consisting of weights of the form 
\begin{align}\label{lessgeneralweight}
	h(z)=\left|\prod_{k=1}^s\left(\frac{z-a_k}{1-\cj{a}_kz}\right)^{m_k}\right|^{p},\qquad p>0, \  m_k\in \N,
\end{align}
is of particular importance because for these weights there is a connection between the reproducing kernel and the contractive zero divisors of the Bergman space $A^p$ (see Section \ref{divisors} below).
   
The following notation will be used throughout the paper. For all non-negative integers $k,j$, we  set 
\begin{align*}
		K^{(k,j)}(z,\zeta):=\frac{\partial^{k+j}}{\partial z^k\partial \cj{\zeta}^{j}} K(z,\zeta).
\end{align*}
Note that the equality 
\[
	K^{(k,j)}(z,\zeta)=	\cj{K^{(j,k)}(\zeta,z)}
\]
follows from the symmetry property \eqref{symmetryproperty}, and that differentiating \eqref{reproducingproperty} yields
\begin{align}\label{reproducingpropertyderivative}
	f^{(j)}(z)=\int_{\D}f(\zeta)K^{(j,0)}(z,\zeta)h(\zeta)d\sigma(\zeta), \quad f\in A^2_h,\ z\in \D.
\end{align}

We write 
\begin{align}\label{qandq*}
q(z):=\prod_{k=1}^s(z-a_k),\quad q^*(z):=z^s\cj{q(1/\cj{z})}=\prod_{k=1}^s(1-\cj{a}_kz),
\end{align}
and for each $1\leq k\leq s$, set 
\[
q_k(z):=\underset{j\not=k}{\prod_{1\leq j\leq s}}(z-a_j),\quad q^*_k(z):=\underset{j\not=k}{\prod_{1\leq j\leq s}}(1-\cj{a}_jz)\,.
\]

Let $c_0,\ldots,c_s$ denote the coefficients of $q$, and for each $1\leq k\leq s$, let $c_{k,0},\ldots, c_{k,s-1}$ denote the coefficients of $q_k$, that is, 
\begin{align*}
q(z)=\sum_{\ell=0}^sc_\ell z^\ell,\qquad q_k(z)=\sum_{\ell=0}^{s-1}c_{k,\ell} z^\ell,\qquad  c_s=c_{k,s-1}=1. 
\end{align*}
From these coefficients, we construct the polynomials 
\begin{align}\label{Tkdef}
	\begin{split}
		T_k(z):={}  &z^k\sum_{\ell=0}^{s-2-k}\cj{c_{s-\ell}}z^{\ell},\qquad 0\leq k\leq s-2,
	\end{split}
\end{align}
and 
\begin{align}\label{Lkdef}
		L_k(z,\zeta):={} &\sum_{r=0}^{s-2}	\left(z^r\sum_{\ell=0}^{s-2-r}\cj{c_{k,s-1-\ell}}z^\ell\right)\left(\cj{\zeta}^r\sum_{\ell=0}^{s-2-r}c_{k,s-1-\ell}\cj{\zeta}^\ell\right),
\end{align} 
$1\leq k\leq s$. When $s=1$, the right-hand side of \eqref{Lkdef} acquires the form   $\sum_{r=0}^{-1}\cdots$. This and other empty sums that later occur are to be interpreted as zero.

Our main result is the following formula.

\begin{theorem}\label{maintheorem} The kernel $K(z,\zeta)$ corresponding to the weight \eqref{generalweight} can be written as  
	\begin{align}\label{mainformula}
		K(z,\zeta)=	&\frac{1}{(1-z\cj{\zeta})^2}+\sum_{k=1}^s\frac{\frac{p_k}{2}(1-|a_k|^2)}{(1-z\cj{\zeta})(1-z\cj{a}_k)(1-a_k\cj{\zeta})}+\frac{J(z,\zeta)}{\cj{q^*(\zeta)}q^*(z)},
	\end{align}
where 
\begin{align}\label{defH}\begin{split}
	J(z,\zeta)={}	&\sum_{k=0}^{s-2}\sum_{j=0}^{s-2}\frac{K^{(k,j)}(0,0)}{k!\, j!}T_k(z)\cj{T_j(\zeta)}-\sum_{k=0}^{s-2}(k+1)T_k(z)\cj{T_k(\zeta)}\\
	&-\sum_{k=1}^{s}\frac{p_k}{2}(1-|a_k|^2)L_k(z,\zeta).
\end{split}
\end{align}	
\end{theorem}

For $s=1$, $J(z,\zeta)\equiv 0$ and \eqref{mainformula} reduces to  
	\begin{align*}
	K(z,\zeta)=	&\frac{1}{(1-z\cj{\zeta})^2}+\frac{\frac{p_1}{2}(1-|a_1|^2)}{(1-z\cj{\zeta})(1-z\cj{a}_1)(1-a_1\cj{\zeta})}.
\end{align*}

For $s=2$, \eqref{mainformula} yields 
	\begin{align}\label{Hansboformula}\begin{split}
	K(z,\zeta)	={} & \frac{1}{(1-z\cj{\zeta})^2}+\sum_{k=1}^2\frac{\frac{p_k}{2}(1-|a_k|^2)}{(1-z\cj{\zeta})(1-z\cj{a}_k)(1-a_k\cj{\zeta})}\\
	& +\frac{K(0,0)-1-\frac{p_1}{2}(1-|a_1|^2)-\frac{p_2}{2}(1-|a_2|^2)}{(1-\cj{a}_1z)(1-\cj{a}_2z)(1-a_1\cj{\zeta})(1-a_2\cj{\zeta})}. 
	\end{split}
\end{align}

Formula \eqref{Hansboformula} was obtained by Hansbo in \cite[Thm. 2.7]{JH} by means of a two-step process that initially uses an iterative procedure to prove the formula for even integer values of $p_1$ and $p_2$, and then transitions from integer to arbitrary  values of $p_1$ and $p_2$ by way of a function theoretical argument involving Carlson's Theorem. Hansbo's method seems to become too intricate when higher values of $s$ are considered, but for $s=2$ it yields the following elegant representation for $K(0,0)$: 
\[
K(0,0)=1+\sum_{k=1}^2\frac{p_k}{2}(1-|a_k|^2)+|a_1-a_2|^2\frac{\sum_{j=1}^\infty j\binom{p_1/2}{j}\binom{p_2/2}{j}\lambda^{j}}{\sum_{j=0}^\infty \binom{p_1/2}{j}\binom{p_2/2}{j}\lambda^{j}},
\]
with $\lambda=(1-|a_1|^2)(1-|a_2|^2)/|1-a_1\cj{a}_2|^2$. We do not know whether this representation generalizes in some satisfactory way to all the derivatives of the kernel in \eqref{defH}, but at least in a structural sense, formula \eqref{mainformula} extends Hansbo's \eqref{Hansboformula} to an arbitrary $s$. 

Besides Hansbo's work, the paper by MacGregor and Stessin \cite{MGS} established a couple of structural formulas for the kernel in the context of Bergman spaces $A^p$, for weights $h$ of the form \eqref{lessgeneralweight} with simple zeros (i.e., $m_k=1$, $1\leq k\leq s$). These formulas equally extend to the more general weight \eqref{generalweight} with little  change in their statements. First, the canonical representation (13) of \cite{MGS} can now be obtained as a corollary of Theorem \ref{maintheorem} as follows. Since the polynomials $q^*_1,\ldots, q^*_s$ are linearly independent, $J(z,\zeta)$ can be expanded as a sum
\[
J(z,\zeta)=\sum_{1\leq k,j\leq s}d_{k,j}q^*_k(z)\cj{q^*_j(\zeta)}\,,
\]
where the coefficients $d_{k,j}$ are given by 
\[
d_{k,j}=\frac{J(1/\cj{a}_k,1/\cj{a}_j)}{q^*_k(1/\cj{a}_k)\cj{q^*_j(1/\cj{a}_j)}}
\]
(if some $a_k$ equals zero, the formula for $d_{k,j}$ has to be adjusted). This gives 
	\begin{align}\label{canonicalrepresentation}
		\begin{split}
	K(z,\zeta)={}	&\frac{1}{(1-z\cj{\zeta})^2}+\sum_{k=1}^s\frac{\frac{p_k}{2}(1-|a_k|^2)}{(1-z\cj{\zeta})(1-z\cj{a}_k)(1-a_k\cj{\zeta})}\\
	&+\sum_{1\leq k,j\leq s}\frac{d_{k,j}}{(1-z\cj{a}_k)(1-a_j\cj{\zeta})}\,.
	\end{split}
\end{align}
Taking  $p_k=p$ in \eqref{canonicalrepresentation} yields an equivalent form of \cite[Eq. (13)]{MGS}. 

The second representation \cite[Eq. (7)]{MGS} that MacGregor and Stessin provide is in terms of the Blaschke products 
\begin{align}\label{Blaschkeproduct}
B(z):=\frac{q(z)}{q^*(z)}, \quad B_k(z):=\frac{q_k(z)}{q^*_k(z)},
\end{align}
and the values of the kernel at the points $a_k$. It results from taking $p_k=p$ in the following formula, more generally valid for a weight $h$ as in \eqref{generalweight}: 
\begin{align}\label{GSformula}
		\begin{split}
			K(z,\zeta)={} &\frac{B(z)\cj{B(\zeta)}}{(1-\cj{\zeta}z)^2}+\sum_{k=1}^s\frac{(\frac{p_k}{2}+1)B_k(z)\cj{B_k(\zeta)}(\cj{\zeta}z-|a_k|^2)}{(1-\cj{\zeta}z)(1-\cj{a}_kz)(1-a_k\cj{\zeta})}\\
			&+\sum_{k=1}^s\sum_{j=1}^s\frac{(1-|a_k|^2)(1-|a_j|^2)B_k(z)\cj{B_j(\zeta)}}{(1-\cj{a}_kz)(1-a_j\cj{\zeta})B_k(a_k)\cj{B_j(a_j)}}K(a_k,a_j)\,.
		\end{split}
	\end{align}

We prove \eqref{GSformula} by following  the original argument in \cite{MGS}, just noting that that argument works equally well for a weight $h$ as in \eqref{generalweight}. The first part of the proof of Theorem \ref{maintheorem} (i.e., up to  \eqref{secondkernelformula}) is also based on the same idea, but applied in a way in which values at the origin are involved instead of values at the points $a_k$. The second part of the proof of Theorem \ref{maintheorem} is more demanding, and consists of performing the algebraic manipulations needed to get the kernel written in the reduced and symmetrical form \eqref{mainformula}.

\subsection{Contractive zero divisors}\label{divisors}

For a fixed $0<p<\infty$, let $A^p$ denote the Bergman space of functions $f$ analytic in $\D$ and such that 
\[
\|f\|_p:=\left(\int_{\D}|f(\zeta)|^pd\sigma(\zeta)\right)^{1/p}<\infty.
\]

Any zero set of a function in $A^p$ (each zero counted as many times as its multiplicity) is called an $A^p$ zero set. For each $A^p$ zero set $Z$, there exists a unique (up to a unimodular constant factor) function $G\in A^p$ of norm $\|G\|_p=1$  such that for every $f\in A^p$ that vanishes on $Z$, the function $f/G\in A^p$ and $\|f/G\|_p\leq \|f\|_p$.  This function $G$ is called the contractive zero divisor (or canonical divisor) of the zero set $Z$. We refer the reader to \cite[Ch. 5]{DSch} for the relevant theory and history. 

 For a finite zero set, say consisting of the points $a_1,\ldots,a_s $, each $a_k$ of multiplicity $m_k$, the canonical divisor is given by \cite[Ch. 5, Thm. 4]{DSch}
\begin{align}\label{canonicaldivisor1}
G(z)=K(0,0)^{-1/p}K(z,0)^{2/p}\prod_{k=1}^s\left(\frac{z-a_k}{1-\cj{a}_kz}\right)^{m_k}\,,
\end{align}
where $K(z,\zeta)$ is the reproducing kernel for the weight \eqref{lessgeneralweight}. If we therefore evaluate \eqref{mainformula} at $\zeta=0$, we get a structural formula for the canonical divisor, the evaluation resulting in 
\begin{align}\label{canonicaldivisor}
	\begin{split}
	K(z,0)={}	&\frac{\cj{q(0)}z^s}{q^*(z)}+\frac{z^{s-1}}{q^*(z)}\sum_{k=1}^{s}\cj{q_k(0)}\left(1+\frac{p_k}{2}(1-|a_k|^2)\right)\\
	&+\frac{1}{q^*(z)}\sum_{k=0}^{s-2}\frac{K^{(k,0)}(0,0)}{k!}T_k(z),
	\end{split}
\end{align}
where (in this context) $p_k=m_kp$, $1\leq k\leq s$. 

We remark that, indeed, to get \eqref{canonicaldivisor} one does not need the reduced symmetric  representation \eqref{mainformula}. Evaluating at $\zeta=0$ in the less refined version   \eqref{firstkernelformula} provides a shorter proof.  

Since the representation \eqref{canonicaldivisor1} already involves values of the kernel at $z=0$ and $\zeta=0$, it is  particularly fitting that \eqref{canonicaldivisor} is equally expressed in terms of derivatives of the kernel at the origin.

\subsection{Computing the values $K^{(k,j)}(0,0)$}
Let us now see how the values $K^{(k,j)}(0,0)$ that occur in \eqref{defH} can be determined by solving  certain linear systems of equations that arise from integration. 
   
Note that the polynomial $T_k$ defined in \eqref{Tkdef} can be characterized as the unique polynomial of degree $\leq s-2$ such that 
\begin{align*}
	\left[\frac{T_k}{q^*}\right]^{(\ell)}(0)=\delta_{k,\ell}k!,\qquad 0\leq \ell\leq s-2.
\end{align*}

From  \eqref{secondkernelformula} and \eqref{defQ} we find  
\begin{align}\label{kernelderivativeformula}
	\sum_{j=0}^{s-2}\frac{K^{(k,j)}(0,0)}{k! j!}\frac{\cj{T_j(\zeta)}}{\cj{q^*(\zeta)}}={} \frac{K^{(k,0)}(0,\zeta)}{k!}-\frac{\cj{P_k(\zeta)}}{\cj{q^*(\zeta)}},\qquad 0\leq k\leq s-2,
\end{align}
where 
\begin{align*}
\cj{P_n(\zeta)}={} &\frac{\cj{\zeta}^{s-1}}{n!}\left.\frac{\partial^{n} }{\partial z^{n}}\left(\frac{\partial}{\partial z}\left(\frac{q(z)}{1-z\cj{\zeta}}\right)+\frac{1}{1-z\cj{\zeta}}\sum_{k=1}^s \frac{\frac{p_k}{2}(1-|a_k|^2)q_k(z)}{1-\cj{a}_kz}\right)\right|_{z=0} \,.
\end{align*}
After computing the derivatives, we can alternatively write 
\[
\cj{P_n(\zeta)}=(n+1)q(0)\cj{\zeta}^{s+n}+\cj{\zeta}^{s-1}\sum_{j=0}^{n}b_{n,j}\cj{\zeta}^j\,,
\]
with 
\begin{align*}
	b_{n,j}={} & (n+1)c_{n+1-j}+\sum_{k=1}^{s}\frac{p_k}{2}(1-|a_k|^2)\sum_{r=0}^{n-j}	(\cj{a}_k)^{r}c_{k,n-j-r}\,.
\end{align*}

By \eqref{reproducingpropertyderivative}, we have that for all $0\leq n\leq s-2$,
\begin{align*}
	\int_{\D}\frac{T_n(\zeta)}{q^*(\zeta)}K^{(k,0)}(0,\zeta)h(\zeta)d\sigma(\zeta)={} &	\left[\frac{T_n}{q^*}\right]^{(k)}(0)=\delta_{n,k}n!\,,
\end{align*}
so that after multiplying \eqref{kernelderivativeformula} by $T_n(\zeta)/q^*(\zeta)$ and integrating with respect to $hd\sigma$, we obtain
\begin{align*}
 \sum_{j=0}^{s-2}\frac{K^{(k,j)}(0,0)}{k! j!}\langle T_n,T_j\rangle=\delta_{n,k}-\langle T_n,P_k\rangle,\qquad 0\leq n\leq s-2,
\end{align*}
with $\langle f,g \rangle:=\int_{\D}f\cj{g}|q^*|^{-2}hd\sigma$.  Since the polynomials  $T_k$ are linearly independent, the Gram matrix $\left(\langle T_n,T_j\rangle\right)_{n,j=0}^{s-2}$ has a non-zero determinant, and so for each $0\leq k\leq s-2$, the vector $(K^{(k,j)}(0,0))_{j=0}^{s-2}$ is the solution of a non-singular linear system of equations whose coefficients and constants can be numerically computed by integration.

 \section{Proofs}
 Both Theorem \ref{maintheorem} and Formula \eqref{GSformula} will be proved with the help of Green's Formula 
 \[
 \frac{1}{2\pi i}\int_{\partial \Omega} g(\zeta)d\zeta=\int_{\Omega}\frac{\partial g}{\partial \cj{\zeta}}\,d\sigma(\zeta)\,,
 \]
 valid for $\Omega$ a finitely connected domain bounded by analytic Jordan curves and $g\in C^1(\cj{\Omega})$.

The proof of Theorem \ref{maintheorem} relies on many heavy computations that at times require the use of several identities. To make it easier to read, we first gather in a lemma a few of these identities that involve the polynomials $q$ and $q^*$ previously defined in \eqref{qandq*}, as well as the polynomial\begin{align*}
  {q'}^*(t):=t^{s-1}\cj{q'(1/\cj{t})}.
\end{align*} 

Note that since $q(z)=\sum_{j=0}^sc_j z^j$ and $q^*(z)=z^s\cj{q(1/\cj{z})}=\sum_{j=0}^s\cj{c_{s-j}}z^j$, we have
\begin{align}\label{coefficientsasTaylor}
c_j=\frac{q^{(j)}(0)}{j!} , \quad \cj{c_{s-j}}=\frac{{q^*}^{(j)}(0)}{j!},\quad 0\leq j\leq  s.
\end{align}

 \begin{lemma} The following identities hold true:
  	\begin{align}\label{InsideD}
 	\begin{split}		&\frac{\cj{\zeta}(1-z\cj{\zeta})(q^*q)(t)}{(z-t)(1-t\cj{\zeta})^2}+\frac{(1-z\cj{\zeta})(q^*q')(t)}{(z-t)(1-t\cj{\zeta})}\\
 		&={} \frac{\cj{\zeta}(q^*q)(t)}{(z-t)(1-t\cj{\zeta})}-	\frac{\partial}{\partial t}\left(\frac{\cj{\zeta}(q^*q)(t)}{(1-t\cj{\zeta})}\right)+  	\frac{(q^*q')(t)}{(z-t)} +\frac{\cj{\zeta}({q^*}'q)(t)}{(1-t\cj{\zeta})},
 	\end{split}
 \end{align}
 
 	\begin{align}\label{starsandderivatives}
 		t{q^*}'(t)=sq^*(t)-{q'}^*(t),
 	\end{align}
  	\begin{align}\label{firsttermJ1}
 	z^s	\cj{(q^*q)(\zeta)}={} &z^s\sum_{\ell=0}^{2s}\frac{(q^*q)^{(\ell)}(0)}{\ell !}\cj{\zeta}^{2s-\ell},
 	\end{align}	
 \begin{align}\label{newidentity5}
 \frac{(q^*q)^{(\ell)}(0)}{\ell !}=\frac{\cj{(q^*q)^{(2s-\ell)}(0)}}{(2s-\ell) !},\qquad 0\leq \ell\leq 2s,
 \end{align}

 \begin{align}\label{newidentity2}
 	z^{s-1}\cj{(q^*q')(\zeta)} ={}	z^{s-1}\sum_{\ell=0}^{2s-1}\frac{\cj{(q^*q')^{(\ell)}(0)}}{\ell !}\cj{\zeta}^{\ell},
 \end{align} 
 	\begin{align}\label{newidentity1}
 	\frac{(q{q'}^*)^{(\ell)}(0)}{\ell !}=\frac{\cj{(q^*q')^{(2s-1-\ell)}(0)}}{(2s-1-\ell) !},\quad 0\leq \ell\leq 2s-1,
 	\end{align}
  \begin{align}\label{newidentity4}
  	\frac{(q^*q)^{(\ell)}(0)}{\ell !}=\sum_{k=0}^{\ell}c_{k}\cj{c_{s-\ell+k}},\qquad 0\leq \ell
  	\leq s,
  \end{align}
  \begin{align}\label{newidentity3}
 	\frac{(q^*q')^{(\ell)}(0)}{\ell !}={}\sum_{j=1}^{\ell+1}jc_{j}\cj{c_{s-1-\ell+j}}\,,\quad 0\leq \ell\leq s-1.
 \end{align}

  \end{lemma}

 \begin{proof} Using that
 	\begin{align*}	
 		\frac{\partial}{\partial t}\left(\frac{\cj{\zeta}(q^*q)(t)}{(1-t\cj{\zeta})}\right)=\frac{\cj{\zeta}^2(q^*q)(t)}{(1-t\cj{\zeta})^2}+\frac{\cj{\zeta}({q^*}'q)(t)}{(1-t\cj{\zeta})}+\frac{\cj{\zeta}(q^*q')(t)}{(1-t\cj{\zeta})},
 	\end{align*}
 	one easily verifies the relation
 	\begin{align*}
 	 		\frac{\cj{\zeta}(1-z\cj{\zeta})(q^*q)(t)}{(z-t)(1-t\cj{\zeta})^2}={} &\frac{\cj{\zeta}(q^*q)(t)}{(z-t)(1-t\cj{\zeta})}+ \frac{\cj{\zeta}(q^*q')(t)}{(1-t\cj{\zeta})}+\frac{\cj{\zeta}({q^*}'q)(t)}{(1-t\cj{\zeta})}-	\frac{\partial}{\partial t}\left(\frac{\cj{\zeta}(q^*q)(t)}{(1-t\cj{\zeta})}\right).
 	\end{align*}
 
On the other hand, we clearly have  
 	\begin{align*}
 		\frac{(1-z\cj{\zeta})(q^*q')(t)}{(z-t)(1-t\cj{\zeta})}={} &	\frac{(q^*q')(t)}{(z-t)}-		\frac{\cj{\zeta} (q^*q')(t)}{(1-t\cj{\zeta})},
 	\end{align*}
and we get \eqref{InsideD} by adding these two latter equalities.

 	Since $q^*(t)=t^s\cj{q(1/\cj{t})}$, we have 
	\begin{align*}
	t{q^*}'(t)={} &t\left(t^s\cj{q(1/\cj{t})}\right)'=t\left(st^{s-1}\cj{q(1/\cj{t})}+t^{s}\cj{q'(1/\cj{t})}(-t^{-2})\right),
\end{align*}	
which is equal to the right-hand side of \eqref{starsandderivatives}.

We next make use of the Maclaurin expansions 
\begin{align}
(q^*q)(\zeta)={} &\sum_{\ell=0}^{2s}\frac{(q^*q)^{(\ell)}(0)}{\ell !}\zeta^\ell\,,\label{taylor-1}\\
	(q^*q')(\zeta)={} &\sum_{\ell=0}^{2s-1}\frac{(q^*q')^{(\ell)}(0)}{\ell !}\zeta^{\ell},\label{taylor0}\\
	(qq'^*)(\zeta)={} &\sum_{\ell=0}^{2s-1}\frac{(qq'^*)^{(\ell)}(0)}{\ell !}\zeta^\ell\label{taylor1}\,.
\end{align}

Using \eqref{taylor-1} we get 
\begin{align}\label{newidentity6}
(q^*q)(\zeta)=\zeta^{2s}\cj{(q^*q)(1/\cj{\zeta})}=\sum_{\ell=0}^{2s}\frac{\cj{(q^*q)^{(\ell)}(0)}}{\ell !}\zeta^{2s-\ell}.
\end{align}
The identity \eqref{firsttermJ1} arises by taking the conjugate in \eqref{newidentity6}, while \eqref{newidentity5} is derived by comparing the Taylor coefficients of \eqref{taylor-1} with those of \eqref{newidentity6}.

 Taking conjugates in \eqref{taylor0} yields \eqref{newidentity2}. But we can also use \eqref{taylor1} to get
 \begin{align}\label{taylor2}
\cj{(q^*q')(\zeta)}=\cj{\zeta}^{2s-1}(q{q'}^*)(1/\cj{\zeta})=\sum_{\ell=0}^{2s-1}\frac{(qq'^*)^{(\ell)}(0)}{\ell !}\cj{\zeta}^{2s-1-\ell},
\end{align}
so that \eqref{newidentity1} follows by comparing the Taylor coefficients of \eqref{newidentity2} and \eqref{taylor2}. 

To prove \eqref{newidentity4} and \eqref{newidentity3}, we  first use Leibnitz's product rule to expand the derivatives $(q^*q)^{(\ell)}(0)$ and $(q^*q')^{(\ell)}(0)$, and then rewrite the resulting sums in terms of the coefficients $c_j$ by means of \eqref{coefficientsasTaylor}.

\end{proof}

\subsection{Proof of Theorem \ref{maintheorem}}
Let $f$ be a function analytic in the closed unit disk  $\cj{\D}$. By an application of the residue theorem, we get that for all $z\in \D$,
\begin{align}\label{Cauchyformula}
	\begin{split}
		f(z)={} &	\frac{z^{s-1}}{q^*(z)}	\frac{1}{2\pi i}\int_{|\zeta|=1}\frac{f(\zeta)q^*(\zeta)}{\zeta^{s-1}(\zeta-z)}d\zeta+\frac{1}{q^*(z)}\sum_{j=0}^{s-2}\frac{f^{(j)}(0)}{j!}T_j(z)\,,
	\end{split}
\end{align}
where
\begin{align}\label{defTk}
T_k(z)=	\frac{k!\,z^{s-1}}{(s-2)!}\binom{s-2}{k}\left[\frac{\partial^{s-2-k}}{\partial t^{s-2-k}}\left(\frac{q^*(t)}{z-t}\right)\right]_{t=0}\,.
\end{align}

Making use of \eqref{coefficientsasTaylor}, this can written as 
 \begin{align*}
	\begin{split}
		T_k(z)={}  &\frac{k! z^{s-1}}{(s-2)!}\binom{s-2}{k}\sum_{\ell=0}^{s-2-k}\binom{s-2-k}{\ell}{q^*}^{(\ell)}(0)\frac{(s-2-k-\ell)!}{ z^{s-1-k-\ell}}\\
		={}  &z^k\sum_{\ell=0}^{s-2-k}\cj{c_{s-\ell}}z^{\ell}.
	\end{split}
\end{align*}
Note that this is the same $T_k$ that was previously defined in \eqref{Tkdef}.

We now seek to express the integral in \eqref{Cauchyformula} as an integral over $\D$ with respect to the weight $h$ given by \eqref{generalweight}. On the unit circle $|\zeta|=1$, $h$ is constant $1$, and also $\cj{\zeta}=1/\zeta$. Then, by an application  of  Green's Formula we obtain that for all $z\in \D$, 
\begin{align}\label{greenformula}
	\begin{split}
		\frac{1}{2\pi i}\int_{|\zeta|=1}\frac{f(\zeta)q^*(\zeta)}{\zeta^{s-1}(\zeta-z)}d\zeta={} &	\frac{1}{2\pi i}\int_{|\zeta|=1}\frac{f(\zeta)\prod_{k=1}^s(1-\cj{a}_k\zeta)}{\zeta^{s-1}(\zeta-z)}d\zeta\\
		={} & 	\frac{1}{2\pi i}\int_{|\zeta|=1}\frac{1}{(1-z\cj{\zeta})}\cj{q(\zeta)}h(\zeta)f(\zeta)d\zeta\\
		={} & 	\int_{\D}\frac{\partial}{\partial \cj{\zeta}}\left(\frac{1}{(1-z\cj{\zeta})}\cj{q(\zeta)}h(\zeta)\right)		f(\zeta)d\sigma(\zeta)\,.
	\end{split}
\end{align}
We clarify that to get \eqref{greenformula} we apply Green's Formula in a subregion of $\D$ that lies exterior to  small circles around each $a_k$ for which  $p_k<0$, and then take limits as the radii of these circles go to $0$. 

The function   
\begin{align}\label{defw}
w(z):=\prod_{k=1}^s\left(\frac{z-a_k}{1-\cj{a}_kz}\right)^{p_k/2}
\end{align}
admits a single-valued branch around any point of $\D$ that is not an $a_k$. Since $h(z)=|w(z)|^2$, we have 
\begin{align}\label{partialformula1}
	\begin{split}
		\frac{\partial}{\partial \cj{\zeta}}\left(\frac{\cj{q(\zeta)}h(\zeta)}{1-z\cj{\zeta}}\right)	={} &\frac{z}{(1-z\cj{\zeta})^2}\cj{q(\zeta)}h(\zeta)+	\frac{w(\zeta)}{1-z\cj{\zeta}}\cj{\frac{\partial}{\partial \zeta}\left(q(\zeta)w(\zeta)\right)}\,.
	\end{split}
\end{align}

By direct computation of the derivative using the product rule, we find 
\begin{align}\label{partialformula2}
	\begin{split}
	\frac{\partial}{\partial \zeta}\left(q(\zeta)w(\zeta)\right)={} &w(\zeta)\sum_{k=1}^sq_k(\zeta) \left(1+\frac{p_k}{2}\frac{1-|a_k|^2}{1-\cj{a}_k\zeta}\right).
	\end{split}
\end{align}

From this point onward we will use the abbreviation 
\[
A_k:=\frac{p_k}{2}(1-|a_k|^2),\qquad 1\leq k\leq s.
\]
The actual values of these constants play no role in what remains of the proof. 

The equalities \eqref{partialformula1} and \eqref{partialformula2} then combine into 
\begin{align}\label{partialformula3}
	\begin{split}
	\frac{\partial}{\partial \cj{\zeta}}\left(\frac{\cj{q(\zeta)}h(\zeta)}{1-z\cj{\zeta}}\right)={}  &h(\zeta)Q(z,\zeta),
	\end{split}
\end{align}
with 
\begin{align}\label{defQ}
Q(z,\zeta):=\frac{z\cj{q(\zeta)}}{(1-z\cj{\zeta})^2}+	\frac{1}{(1-z\cj{\zeta})}\sum_{k=1}^s\cj{q_k(\zeta)} \left(1+\frac{A_k}{1-a_k\cj{\zeta}}\right).
\end{align}

If we now write the derivatives in \eqref{Cauchyformula} in the integral form
\[
f^{(j)}(0) =\int_{\D} f(\zeta)K^{(j,0)}(0,\zeta)h(\zeta)d\sigma(\zeta),
\]
then \eqref{Cauchyformula} can be written with the help of \eqref{greenformula} and \eqref{partialformula3} as
\begin{align*}
	\begin{split}
		f(z)={} &\int_{\D}\frac{f(\zeta)}{q^*(z)}\left(z^{s-1}Q(z,\zeta)+\sum_{j=0}^{s-2}\frac{T_j(z)}{j!}K^{(j,0)}(0,\zeta)\right)h(\zeta)d\sigma(\zeta)\,.
	\end{split}
\end{align*}
By a density argument, this reproducing property also holds for every $f\in A^2_h$, and so we must have  
\begin{align}\label{firstkernelformula}
K(z,\zeta)=	\frac{z^{s-1}Q(z,\zeta)}{q^*(z)}+\frac{1}{q^*(z)}\sum_{k=0}^{s-2}\frac{T_k(z)K^{(k,0)}(0,\zeta)}{k!}\,.
\end{align}

By the symmetry of the kernel,
\[
K(z,\zeta)=	\frac{\cj{\zeta^{s-1} Q(\zeta,z)}}{\cj{q^*(\zeta)}}+\frac{1}{\cj{q^*(\zeta)}}\sum_{j=0}^{s-2}\frac{\cj{T_j(\zeta)}K^{(0,j)}(z,0)}{j!}\,.
\]

Taking derivatives with respect to $z$ and evaluating at $z=0$ yields  
\begin{align}\label{secondkernelformula}
	K^{(k,0)}(0,\zeta)={} &	\frac{\cj{\zeta}^{s-1} }{\cj{q^*(\zeta)}}\left[\frac{\partial^k}{\partial z^k}\cj{Q(\zeta,z)}\right]_{z=0}+\frac{1}{\cj{q^*(\zeta)}}\sum_{j=0}^{s-2}\frac{\cj{T_j(\zeta)}K^{(k,j)}(0,0)}{j!}\,.
\end{align}

Substituting \eqref{secondkernelformula} into \eqref{firstkernelformula} we get
\begin{align}\label{formulafirstpart}
	\begin{split}
	K(z,\zeta)={}	&	\frac{z^{s-1}Q(z,\zeta)}{q^*(z)}+\frac{\cj{\zeta}^{s-1}}{\cj{q^*(\zeta)}q^*(z)}\sum_{k=0}^{s-2}\frac{T_k(z)}{k!}\left[\frac{\partial^k}{\partial z^k}\cj{Q(\zeta,z)}\right]_{z=0}\\
	&+\frac{1}{\cj{q^*(\zeta)}q^*(z)}\sum_{k=0}^{s-2}\sum_{j=0}^{s-2} \frac{K^{(k,j)}(0,0)}{k! j!}T_k(z)\cj{T_j(\zeta)}\,.
	\end{split}
\end{align}

This concludes the first part of the proof. The second part amounts to performing the algebraic manipulations needed to reduce \eqref{formulafirstpart} to \eqref{mainformula}.
We begin by writing \eqref{formulafirstpart}  as 
\begin{align}\label{KintermsofJ}
	K(z,\zeta)={}	&\frac{N(z,\zeta)}{(1-z\cj{\zeta})^2q^*(z)\cj{q^*(\zeta)}}+\frac{1}{q^*(z)\cj{q^*(\zeta)}}\sum_{k=0}^{s-2}\sum_{j=0}^{s-2}\frac{K^{(k,j)}(0,0)}{k! j!}T_k(z)\cj{T_j(\zeta)}\,,
\end{align}
with
\begin{align}\label{defJ}
	\begin{split}
	N(z,\zeta):=	{}&	(1-z\cj{\zeta})^2z^{s-1}\cj{q^*(\zeta)}Q(z,\zeta)\\
	&+(1-z\cj{\zeta})^2\cj{\zeta}^{s-1}\sum_{k=0}^{s-2}\frac{T_k(z)}{k!}\left[\frac{\partial^k}{\partial z^k}\cj{Q(\zeta,z)}\right]_{z=0}\,.
	\end{split}
\end{align}

By \eqref{defQ}, and since 
\[
q'(t)=\sum_{k=1}^sq_k(t)\,,
\]
we have
\begin{align}\label{firstpartofJ}
	\begin{split}
	(1-z\cj{\zeta})^2z^{s-1}\cj{q^*(\zeta)}Q(z,\zeta)={} &z^s\cj{(q^*q)(\zeta)}+ (1-z\cj{\zeta})z^{s-1}\cj{(q^*q')(\zeta)} \\
	&+ (1-z\cj{\zeta})z^{s-1}\sum_{k=1}^sA_k\cj{(q_k^*q_k)(\zeta)} \,.
	\end{split}
\end{align}

In view of the representation \eqref{defTk} for $T_k$, we  can write
\begin{align}\label{secondpartofJ}
\sum_{k=0}^{s-2}\frac{T_k(z)}{k!}\left[\frac{\partial^k}{\partial z^k}\cj{Q(\zeta,z)}\right]_{z=0}={}& \frac{z^{s-1}}{(s-2)!}\left[\frac{\partial^{s-2}}{\partial t^{s-2}}\left(\frac{ q^*(t)}{z-t}\cj{Q(\zeta,t)}\right)\right]_{t=0},
\end{align}	
and this latter derivative can be expanded by differentiating the identity  
\begin{align}\label{terminsidesecondpartofJ}
	\frac{q^*(t)}{z-t}\cj{Q(\zeta,t)}={} &\frac{\cj{\zeta}(q^*q)(t)}{(z-t)(1-t\cj{\zeta})^2}+	\frac{(q^*q')(t)}{(z-t)(1-t\cj{\zeta})} +\frac{\sum_{k=1}^sA_k(q^*_kq_k)(t)}{(z-t)(1-t\cj{\zeta})},
\end{align}
which is just a consequence of \eqref{defQ}.

The relations \eqref{terminsidesecondpartofJ}, \eqref{secondpartofJ}, \eqref{firstpartofJ} and \eqref{defJ}  then allow us to write $N(z,\zeta)$ in the form 
\begin{align}\label{representationJ}
	\begin{split}
N(z,\zeta)={}&z^s\cj{(q^*q)(\zeta)}+ (1-z\cj{\zeta})z^{s-1}\cj{(q^*q')(\zeta)}+(1-z\cj{\zeta})(z\cj{\zeta})^{s-1}D(z,\zeta)\\
	&+(1-z\cj{\zeta})z^{s-1}\sum_{k=1}^sA_k\cj{(q_k^*q_k)(\zeta)} \\
	&+\frac{(1-z\cj{\zeta})^2z^{s-1}\cj{\zeta}^{s-1}}{(s-2)!}\sum_{k=1}^sA_k\left[\frac{\partial^{s-2}}{\partial t^{s-2}} 	\left(\frac{(q^*_kq_k)(t)}{(z-t)(1-t\cj{\zeta})}\right)\right]_{t=0}\,,
	\end{split}
\end{align}
where
\[
D(z,\zeta):=	\frac{1}{(s-2)!}\left[\frac{\partial^{s-2}}{\partial t^{s-2}}\left(\frac{(1-z\cj{\zeta})\cj{\zeta}(q^*q)(t)}{(z-t)(1-t\cj{\zeta})^2}+	\frac{(1-z\cj{\zeta})(q^*q')(t)}{(z-t)(1-t\cj{\zeta})}\right)\right]_{t=0}\,.
\]

We now utilize the identity \eqref{InsideD} and the fact that derivatives at the origin of any function $f$ satisfy
\[
f^{(m)}(0)=\frac{1}{m+1}\left.\left(tf(t)\right)^{(m+1)}\right|_{t=0}\,,
\]
to get $D(z,\zeta)$ written in the form
\begin{align}\label{formforD}
	\begin{split}
&\frac{1}{(s-1)!}\left[\frac{\partial^{s-1}}{\partial t^{s-1}}\left(\frac{t\cj{\zeta }(q^*q)(t)}{(z-t)(1-t\cj{\zeta})}-(s-1)\frac{\cj{\zeta}(q^*q)(t)}{1-t\cj{\zeta}}+\frac{t\cj{\zeta}({q^*}'q)(t)}{1-t\cj{\zeta}}\right)\right]_{t=0}\\
&+\frac{1}{(s-2)!}\left[\frac{\partial^{s-2}}{\partial t^{s-2}}\left(\frac{(q^*q')(t)}{z-t}\right)\right]_{t=0}\,.
\end{split}
\end{align}

By rearranging terms in \eqref{formforD} and making use of \eqref{starsandderivatives} we obtain
\begin{align*}D(z,\zeta) ={} & \frac{z\cj{\zeta}}{(s-1)!}\left[\frac{\partial^{s-1}}{\partial t^{s-1}}\left(\frac{(q^*q)(t)}{(z-t)(1-t\cj{\zeta})}\right)\right]_{t=0}\\
	&-\frac{1}{(s-1)!}\left[\frac{\partial^{s-1}}{\partial t^{s-1}}\left(\frac{\cj{\zeta}({q'}^*q)(t)}{1-t\cj{\zeta}}\right)\right]_{t=0}\\
	&+\frac{1}{(s-2)!}\left[\frac{\partial^{s-2}}{\partial t^{s-2}}\left(\frac{(q^*q')(t)}{z-t}\right)\right]_{t=0}\,,
\end{align*}
and after plugging this expression for $D(z,\zeta)$ in \eqref{representationJ}, we see that we can represent $N(z,\zeta)$ as a sum
\begin{align}\label{definitionofJ}
	N(z,\zeta)={}&	N_1(z,\zeta)+	 (1-z\cj{\zeta})N_2(z,\zeta)+ (1-z\cj{\zeta})\sum_{k=1}^sA_kM_k(z,\zeta)\,,
\end{align}
with 
\begin{align}\label{J1def}
N_1(z,\zeta):={}&z^s\cj{(q^*q)(\zeta)}+ \frac{(z\cj{\zeta})^{s}}{(s-1)!}\left[\frac{\partial^{s-1}}{\partial t^{s-1}}\left(\frac{(1-z\cj{\zeta})(q^*q)(t)}{(z-t)(1-t\cj{\zeta})}\right)\right]_{t=0}\,,
\end{align}
\begin{align}\label{defJ2}
	\begin{split}
	N_2(z,\zeta):={}&z^{s-1}\cj{(q^*q')(\zeta)} + \frac{(z\cj{\zeta})^{s-1}}{(s-2)!}\left[\frac{\partial^{s-2}}{\partial t^{s-2}}\left(\frac{(q^*q')(t)}{z-t}\right)\right]_{t=0}\\
	&-\frac{(z\cj{\zeta})^{s-1}}{(s-1)!}\left[\frac{\partial^{s-1}}{\partial t^{s-1}}\left(\frac{\cj{\zeta}({q'}^*q)(t)}{1-t\cj{\zeta}}\right)\right]_{t=0}\,,
	\end{split}
\end{align}
and 
\begin{align*}
	\begin{split}
		M_k(z,\zeta):={}&z^{s-1}\cj{(q_k^*q_k)(\zeta)}+\frac{(z\cj{\zeta})^{s-1}}{(s-2)!}\left[\frac{\partial^{s-2}}{\partial t^{s-2}} 	\left(\frac{(1-z\cj{\zeta})(q^*_kq_k)(t)}{(z-t)(1-t\cj{\zeta})}\right)\right]_{t=0}\,.
	\end{split}
\end{align*}

We next show that 
\begin{align}\label{firstexpressionforJ1}
	N_1(z,\zeta)={} &	q^*(z)\cj{q^*(\zeta)}-(1-z\cj{\zeta})L(z,\zeta)\,,
\end{align}  
with
\begin{align}\label{simplifiedformofL}
	L(z,\zeta):={} &\sum_{r=0}^{s-1}	(z\cj{\zeta})^r\left(\sum_{j=0}^{s-1-r}c_{s-j}\cj{\zeta}^j\right)\left(\sum_{\ell=0}^{s-1-r}\cj{c_{s-\ell}}z^\ell\right)\,.
\end{align} 
Since  $N_1(z,\zeta)$ and $M_k(z,\zeta)$ have the  same structure, this will simultaneously  show that 
\begin{align}\label{expressionforHk}
M_k(z,\zeta)=	q_k^*(z)\cj{q_k^*(\zeta)}-(1-z\cj{\zeta})L_k(z,\zeta),\quad 1\le k\leq s,
\end{align} 
with $L_k$ given by \eqref{Lkdef}.

With the help of the identity 
\[
\frac{1-z\cj{\zeta}}{(z-t)(1-t\cj{\zeta})}=\frac{1}{z-t}-\frac{\cj{\zeta}}{1-t\cj{\zeta}}\,,
\]
the second term in the right-hand side of \eqref{J1def} can be written as  
\begin{align}\label{secondtermJ1}
	\cj{\zeta}^{s}\sum_{\ell=0}^{s-1}\frac{(q^*q)^{(\ell)}(0)}{\ell !}z^{\ell}-	z^{s}\sum_{\ell=0}^{s-1}\frac{(q^*q)^{(\ell)}(0)}{\ell !}\cj{\zeta}^{2s-\ell}.
\end{align}  

To simplify the upcoming presentation, we write
\[
C_\ell:=\frac{(q^*q)^{(\ell)}(0)}{\ell !},\quad 0\leq \ell
\leq 2s\,,
\]
so that by \eqref{newidentity5},
\[
\cj{C_\ell}=\frac{(q^*q)^{(2s-\ell)}(0)}{(2s-\ell) !},\quad 0\leq \ell\leq 2s,
\]
and by \eqref{newidentity4},
\[
C_\ell= \sum_{k=0}^{\ell}c_{k}\cj{c_{s-\ell+k}}=\sum_{k=0}^{\ell}\cj{c_{s-k}}c_{\ell-k}\,, \quad 0\leq \ell\leq s.
\]
Note that $
	C_s=\sum_{j=0}^{s}|c_j|^2\geq 1
$.

Since the right-hand side  of \eqref{J1def} is equal to the sum of \eqref{firsttermJ1} and \eqref{secondtermJ1}, we get 
\begin{align}\label{secondexpressionforJ1}
	N_1(z,\zeta)  	&={} C_s(z\cj{\zeta})^s+\cj{\zeta}^{s}\sum_{\ell=0}^{s-1}C_{\ell}z^{\ell}+	z^{s}\sum_{\ell=0}^{s}\cj{C_\ell}\,\cj{\zeta}^{\ell}\,.
\end{align}  

Let us now define, for every integer $n\geq 0$, 
\[
V_{n}(t):=\frac{t^{n+1}-1}{t-1}=\sum_{j=0}^{n}t^{j},
\]
and let us also agree that $V_{-1}\equiv 0$. With the help of these polynomials, we can then write \eqref{secondexpressionforJ1} as
\begin{align}\label{thirdexpressionforJ1}
	N_1(z,\zeta)={}&(z\cj{\zeta}-1)\left(C_{s}V_{s-1}(z\cj{\zeta})+\sum_{\ell=0}^{s-1}C_{\ell}\cj{\zeta}^{s-\ell}V_{\ell-1}(z\cj{\zeta})+\sum_{\ell=0}^{s-1}\cj{C_{\ell}}z^{s-\ell}V_{\ell-1}(z\cj{\zeta})\right)\nonumber\\
&+C_s+\sum_{\ell=0}^{s-1}C_{\ell}\cj{\zeta}^{s-\ell}+ \sum_{\ell=0}^{s-1}\cj{C_{\ell}}z^{s-\ell}\,.
\end{align}  

Letting  $\ell=j-k-1$ be the new counter, one verifies the equalities
\begin{align}\label{terms1}
\sum_{\ell=0}^{s-1}C_{\ell}\cj{\zeta}^{s-\ell}=	\sum_{0\leq k<j\leq s}\cj{c_{s-k}}c_{s-j}\cj{\zeta}^{j-k}\,,
\end{align}  
and
\begin{align}\label{terms2} 
\sum_{\ell=0}^{s-1}C_{\ell}\cj{\zeta}^{s-\ell}V_{\ell-1}(z\cj{\zeta})=	\sum_{0\leq k<j\leq s}\cj{c_{s-k}}c_{s-j}\cj{\zeta}^{j-k}V_{s-1-(j-k)}(z\cj{\zeta})\,.
\end{align}  

After expanding the product 
\[
q^*(z)\cj{q^*(\zeta)}=\left(\sum_{u=0}^s\cj{c_{s-u}}z^u\right)\left(\sum_{j=0}^sc_{s-j}\cj{\zeta}^j\right),
\]
we can  express it by means of the polynomials $V_n$ as
\begin{align}\label{productofq*}\begin{split}
 q^*(z)\cj{q^*(\zeta)}={}&(z\cj{\zeta}-1)\sum_{j=0}^{s}|c_{s-j}|^2V_{j-1}(z\cj{\zeta})   \\
	&+(z\cj{\zeta}-1)\sum_{0\leq k<j\leq s}\left(c_{s-k}\cj{c_{s-j}}z^{j-k}+\cj{c_{s-k}}c_{s-j}\cj{\zeta}^{j-k}\right)V_{k-1}(z\cj{\zeta}) \\
	&+C_s+\sum_{0\leq k<j\leq s}\cj{c_{s-k}}c_{s-j}\cj{\zeta}^{j-k}+\sum_{0\leq k<j\leq s}c_{s-k}\cj{c_{s-j}}z^{j-k}\,.
	\end{split}
\end{align}  

Four terms in \eqref{thirdexpressionforJ1} can be rewritten using \eqref{terms1} and \eqref{terms2}, so that subtracting \eqref{productofq*} from \eqref{thirdexpressionforJ1} yields
\begin{align*}
	N_1(z,\zeta)	={} &	q^*(z)\cj{q^*(\zeta)}+(z\cj{\zeta}-1)\sum_{j=0}^{s-1}|c_{s-j}|^2(z\cj{\zeta})^{j}V_{s-1-j}(z\cj{\zeta})\\
	&+(z\cj{\zeta}-1)\sum_{0\leq k<j\leq s-1}\cj{c_{s-k}}c_{s-j}z^k\cj{\zeta}^{j}V_{s-1-j}(z\cj{\zeta})\\
	&+(z\cj{\zeta}-1)\sum_{0\leq k<j\leq s-1}c_{s-k}\cj{c_{s-j}}z^{j}\cj{\zeta}^{k}V_{s-1-j}(z\cj{\zeta})\,.
\end{align*}  

If we now write these sums in the form 
\begin{align*}
\sum_{j=0}^{s-1}|c_{s-j}|^2(z\cj{\zeta})^{j}V_{s-1-j}(z\cj{\zeta}) ={} &\sum_{j=0}^{s-1}\left(\sum_{\ell=s-j}^{s}|c_\ell|^2 \right)(z\cj{\zeta})^{j}\,,
\end{align*}  
and
\begin{align*}
\sum_{0\leq k<j\leq s-1}\cj{c_{s-k}}c_{s-j}z^k\cj{\zeta}^{j}V_{s-1-j}(z\cj{\zeta})={}&\sum_{0\leq k<j\leq s-1}z^k\cj{\zeta}^{j}\sum_{\ell=0}^{k}\cj{c_{s-k+\ell}}c_{s-j+\ell}\,,
\end{align*}  
then we see that $N_1(z,\zeta)$ can be written as in \eqref{firstexpressionforJ1} with 
\begin{align}\label{secondexpressionforL}
	\begin{split}
	L(z,\zeta)={} &	\sum_{j=0}^{s-1}\left(\sum_{\ell=s-j}^{s}|c_\ell|^2\right)(z\cj{\zeta})^{j}+\sum_{0\leq k<j\leq s-1}\left(\sum_{\ell=0}^{k}\cj{c_{s-k+\ell}}c_{s-j+\ell}\right)z^k\cj{\zeta}^j\\
	&+\sum_{0\leq k<j\leq s-1}\left(\sum_{\ell=0 }^{k}c_{s-k+\ell}\cj{c_{s-j+\ell}}\right)z^j\cj{\zeta}^k\,.
	\end{split}
\end{align}  

To verify that \eqref{secondexpressionforL} coincides with \eqref{simplifiedformofL}, we group the terms corresponding to the lowest values of $\ell$ to get
\begin{align}\label{iterativeformulaforL}
	\begin{split}
	L(z,\zeta)={} &	\left(\sum_{j=0}^{s-1}c_{s-j}\cj{\zeta}^j\right)\left(\sum_{\ell=0}^{s-1}\cj{c_{s-\ell}}z^\ell\right)+ (z\cj{\zeta})\sum_{j=0}^{s-2}\left(\sum_{\ell=s-j}^{s}|c_\ell|^2\right)(z\cj{\zeta})^{j}\\
	&+(z\cj{\zeta})\sum_{0\leq k<j\leq s-2}\left(\sum_{\ell=0}^{k}\cj{c_{s-k+\ell}}c_{s-j+\ell}\right)z^k\cj{\zeta}^j\\
	&+(z\cj{\zeta})\sum_{0\leq k<j\leq s-2}\left(\sum_{\ell=0 }^{k}c_{s-k+\ell}\cj{c_{s-j+\ell}}\right)z^j\cj{\zeta}^k\,.
	\end{split}
\end{align} 
The sum of the last three terms equals the product of $(z\cj{\zeta})$ and a second factor, which is just the sum in \eqref{secondexpressionforL} with the terms corresponding to $j=s-1$ removed. This decomposition can then be iterated to produce \eqref{simplifiedformofL}.

We will also need a further decomposition of $L(z,\zeta)$, namely, writing \eqref{secondexpressionforL} as 
\begin{align}\label{decomposedL}
	\begin{split}
	L(z,\zeta)={} &(z\cj{\zeta}-1)\sum_{j=0}^{s-1}\left(\sum_{\ell=s-j}^{s}|c_\ell|^2\right)V_{j-1}(z\cj{\zeta})\\
	&+(z\cj{\zeta}-1)\sum_{0\leq k<j\leq s-1}\left(\sum_{\ell=0}^{k}\cj{c_{s-k+\ell}}c_{s-j+\ell}\right)\cj{\zeta}^{j-k}V_{k-1}(z\cj{\zeta})\\
	&+(z\cj{\zeta}-1)\sum_{0\leq k<j\leq s-1}\left(\sum_{\ell=0 }^{k}c_{s-k+\ell}\cj{c_{s-j+\ell}}\right)z^{j-k}V_{k-1}(z\cj{\zeta})\\
	&+\sum_{j=0}^{s-1}\left(\sum_{\ell=s-j}^{s}|c_\ell|^2\right)+\sum_{0\leq k<j\leq s-1}\left(\sum_{\ell=0}^{k}\cj{c_{s-k+\ell}}c_{s-j+\ell}\right)\cj{\zeta}^{j-k}\\
	&+\sum_{0\leq k<j\leq s-1}\left(\sum_{\ell=0 }^{k}c_{s-k+\ell}\cj{c_{s-j+\ell}}\right)z^{j-k}\,.
	\end{split}
\end{align}  

 We now focus on the quantity $N_2(z,\zeta)$ defined in \eqref{defJ2}. With the aid of the relations \eqref{newidentity1} and \eqref{newidentity2}, we can expand the derivatives in \eqref{defJ2} to get
\begin{align}\label{secondexpressionforJ2}
	N_2(z,\zeta)	={} &\cj{\zeta}^{s-1}\sum_{\ell=0}^{s-2}\frac{(q^*q')^{(\ell)}(0)}{\ell !}z^{\ell}+	z^{s-1}\sum_{\ell=0}^{s-1}\frac{\cj{(q^*q')^{(\ell)}(0)}}{\ell !}\cj{\zeta}^{\ell}\,,
\end{align}
and by \eqref{newidentity3}, we can then write \eqref{secondexpressionforJ2} as
\begin{align}\label{thirdexpressionforJ2}
	\begin{split}
	N_2(z,\zeta)={} &(z\cj{\zeta}-1)\sum_{\ell=1}^{s}\ell|c_\ell|^2V_{s-2}(z\cj{\zeta})	+\sum_{\ell=1}^{s}\ell|c_\ell|^2\\
	&+(z\cj{\zeta}-1)\sum_{\ell=0}^{s-2}z^{s-1-\ell}V_{\ell-1}(z\cj{\zeta})\sum_{j=1}^{\ell+1}j\cj{c_{j}}c_{s-1-\ell+j}\\
		&+(z\cj{\zeta}-1)\sum_{\ell=0}^{s-2}\cj{\zeta}^{s-1-\ell}V_{\ell-1}(z\cj{\zeta})\sum_{j=1}^{\ell+1}jc_{j}\cj{c_{s-1-\ell+j}}\\
	&+		\sum_{\ell=0}^{s-2}\cj{\zeta}^{s-1-\ell}\sum_{j=1}^{\ell+1}jc_{j}\cj{c_{s-1-\ell+j} }	+\sum_{\ell=0}^{s-2}z^{s-1-\ell}\sum_{j=1}^{\ell+1}j\cj{c_{j}}c_{s-1-\ell+j} \,.
	\end{split}
\end{align}

If we now rewrite the last five summands  of \eqref{thirdexpressionforJ2} in the form 
\[
\sum_{\ell=1}^{s}\ell|c_\ell|^2=\sum_{k=0}^{s-1}\sum_{\ell=s-k}^{s}|c_\ell|^2,
\]
\begin{align*}
\sum_{u=0}^{s-2}\cj{\zeta}^{s-1-u}\sum_{j=1}^{u+1}jc_{j}\cj{c_{s-1-u+j}}={}&	\sum_{0\leq k<j\leq s-1}\left(\sum_{\ell=0}^{k}\cj{c_{s-k+\ell}}c_{s-j+\ell}\right)\cj{\zeta}^{j-k},
\end{align*} 
and 
\begin{align*}
&\sum_{u=0}^{s-2} \cj{\zeta}^{s-1-u}V_{u-1}(z\cj{\zeta})\sum_{j=1}^{u+1}jc_{j}\cj{c_{s-1-u+j}}\\
&\quad 	=	\sum_{0\leq k<j\leq s-1}\left(\sum_{\ell=0}^{k}\cj{c_{s-k+\ell}}c_{s-j+\ell}\right)V_{s-2-j+k}(z\cj{\zeta})\cj{\zeta}^{j-k},
\end{align*}
and then subtract \eqref{decomposedL} from \eqref{thirdexpressionforJ2}, we obtain 
\begin{align}\label{fourthexpresionforJ2}
	\begin{split}
\frac{N_2(z,\zeta)-	L(z,\zeta)}{(z\cj{\zeta}-1)}={} &\sum_{k=0}^{s-2}(z\cj{\zeta})^{k}V_{s-2-k}(z\cj{\zeta})\left(\sum_{\ell=s-k}^{s}|c_\ell|^2\right)\\
	&+	\sum_{0\leq k<j\leq s-2}\left(\sum_{\ell=0}^{k}\cj{c_{s-k+\ell}}c_{s-j+\ell}\right)V_{s-2-j}(z\cj{\zeta})z^{k}\cj{\zeta}^{j}\\	
	&+	\sum_{0\leq k<j\leq s-2}\left(\sum_{\ell=0}^{k}c_{s-k+\ell}\cj{c_{s-j+\ell}}\right)V_{s-2-j}(z\cj{\zeta})z^{j}\cj{\zeta}^{k}\,.
	\end{split}
\end{align}

We then replace the terms in the right-hand side of \eqref{fourthexpresionforJ2} with the equivalent expressions 
\begin{align*}
\sum_{k=0}^{s-2}(z\cj{\zeta})^{k}V_{s-2-k}(z\cj{\zeta})\left(\sum_{\ell=s-k}^{s}|c_\ell|^2\right)={} &\sum_{j=0}^{s-2}(z\cj{\zeta})^{j}\sum_{\ell=0}^{j}(\ell+1)|c_{s-j+\ell}|^2\,,
\end{align*}
\begin{align*}
&\sum_{0\leq k<j\leq s-2}\left(\sum_{\ell=0}^{k}\cj{c_{s-k+\ell}}c_{s-j+\ell}\right)V_{s-2-j}(z\cj{\zeta})z^{k}\cj{\zeta}^{j}\\
&\quad={}\sum_{0\leq k<j\leq s-2}z^{k}\cj{\zeta}^{j}\sum_{\ell=0}^{k}(\ell+1)\cj{c_{s-k+\ell}}c_{s-j+\ell},
\end{align*}
and after doing that, we group the terms corresponding to $\ell=0$ to obtain
\begin{align}\label{iterativeforJ_2}
	\begin{split}
\frac{N_2(z,\zeta)-	L(z,\zeta)}{(z\cj{\zeta}-1)}	={} &\left(\sum_{j=0}^{s-2}c_{s-j}\cj{\zeta}^{j}\right)\left(\sum_{k=0}^{s-2}\cj{c_{s-k}}z^{k}\right)\\
&+(z\cj{\zeta})\sum_{j=0}^{s-3}(z\cj{\zeta})^{j}\sum_{\ell=0}^{j}(\ell+2)|c_{s-j+\ell}|^2\\
		&+(z\cj{\zeta})\sum_{0\leq k<j\leq s-3}z^{k}\cj{\zeta}^{j}\sum_{\ell=0}^{k}(\ell+2)\cj{c_{s-k+\ell}}c_{s-j+\ell}\\
	&+(z\cj{\zeta})\sum_{0\leq k<j\leq s-3}z^{j-1}\cj{\zeta}^{k-1}\sum_{\ell=0}^{k}(\ell+2)c_{s-k+\ell}\cj{c_{s-j+\ell}}\,.
\end{split}
\end{align}

By the same argument we used to derive \eqref{simplifiedformofL} from \eqref{secondexpressionforL}, we can iterate \eqref{iterativeforJ_2} to finally get 
\begin{align*}
N_2(z,\zeta)-	L(z,\zeta)={} &(z\cj{\zeta}-1)	\sum_{r=0}^{s-2}(r+1)(z\cj{\zeta})^r\left(\sum_{j=0}^{s-2-r}c_{s-j}\cj{\zeta}^{j}\right)\left(\sum_{k=0}^{s-2-r}\cj{c_{s-k}}z^{k}\right).
\end{align*}

Inserting \eqref{firstexpressionforJ1}, \eqref{expressionforHk}, and the latter equality into \eqref{definitionofJ} (recall \eqref{Tkdef}) yields 
\begin{align*}
	N(z,\zeta)={}&	q^*(z)\cj{q^*(\zeta)}- (1-z\cj{\zeta})^2\sum_{r=0}^{s-2}(r+1)T_r(z)\cj{T_r(\zeta)}\\
	&+ (1-z\cj{\zeta})\sum_{k=1}^sA_kq_k^*(z)\cj{q_k^*(\zeta)} -(1-z\cj{\zeta})^2\sum_{k=1}^sA_kL_k(z,\zeta).
\end{align*}
In view of \eqref{KintermsofJ}, this  establishes Theorem \ref{maintheorem}.

\subsection{Proof of Formula  \eqref{GSformula}} \label{miscellaneousproofs}
The proof is basically the same given in \cite{MGS}, the argument simply works equally well   for a product of different powers of Blaschke factors. For every $f$ in the disk algebra, a residue calculation gives (this is (2) in \cite{MGS})  
	\begin{align}\label{firstrepresentation}
		\begin{split}
			\frac{1}{2\pi i}\int_{|\zeta|=1}\frac{B(z)}{B(\zeta)}\frac{f(\zeta)}{\zeta-z}d\zeta ={} & f(z)-B(z)\sum_{k=1}^s\frac{1-|a_k|^2}{B_k(a_k)}\frac{f(a_k)}{z-a_k}.
		\end{split}
	\end{align}

Letting $w$ be as in \eqref{defw}, and since $h\equiv 1$ on the unit circle, we can apply  Green's Formula to obtain  
	\begin{align}\label{secondrepresentation}
		\begin{split}
			&	\frac{1}{2\pi i}\int_{|\zeta|=1}\frac{B(z)}{B(\zeta)}\frac{f(\zeta)}{\zeta-z}d\zeta=	\frac{B(z)}{2\pi i}\int_{|\zeta|=1}\frac{\cj{\zeta }}{1-\cj{\zeta}z}\cj{B(\zeta)w(\zeta)}f(\zeta)w(\zeta)d\zeta\\
			&=\int_{\D}\left(\frac{B(z)\cj{B(\zeta)}}{(1-\cj{\zeta}z)^2}+\frac{B(z)\cj{\zeta }}{1-\cj{\zeta}z}\sum_{k=1}^s\frac{(\frac{p_k}{2}+1)(1-|a_k|^2)\cj{B_k(\zeta)}}{(1-a_k\cj{\zeta})^2}\right)f(\zeta)h(\zeta)d\sigma(\zeta)
		\end{split}
	\end{align}

Writing each $f(a_k)$ in \eqref{firstrepresentation} as the integral of $f(\zeta)$ against $K(a_k,\zeta)$ and replacing the left-hand side of \eqref{firstrepresentation} by the last integral in \eqref{secondrepresentation}, we obtain a representation for $f(z)$ as an integral over $\D$ with respect to $h(\zeta)d\sigma(\zeta)$, whose integrand is the product of $f(\zeta)$ and the function 
	\begin{align*}
		\begin{split}
		&\frac{B(z)\cj{B(\zeta)}}{(1-\cj{\zeta}z)^2}+\frac{B(z)\cj{\zeta }}{1-\cj{\zeta}z}\sum_{k=1}^s\frac{(\frac{p_k}{2}+1)(1-|a_k|^2)\cj{B_k(\zeta)}}{(1-a_k\cj{\zeta})^2}\\
			&+\sum_{k=1}^s\frac{(1-|a_k|^2)B_k(z)}{(1-\cj{a}_kz)B_k(a_k)}K(a_k,\zeta).
		\end{split}
	\end{align*}
Since this function is analytic in $z$ and anti-analytic in $\zeta$, it must coincide with $K(z,\zeta)$, and from this expression for the kernel we can continue as in \cite{MGS} to derive \eqref{GSformula}.

\section{Declarations}

\paragraph{Ethical approval:} Not applicable.

\paragraph{Competing interests:} The author has no relevant financial or non-financial interests to disclose. The author has no conflicts of interest to declare that are relevant to the content of this article. The only organization with an interest in the research accomplished in this manuscript is the author's employer, The University of Mississippi. 

\paragraph{Authors' contributions:} Not applicable.

\paragraph{Funding:} The author did not receive funding from any organization for the submitted work.

\paragraph{Availability of data and materials:}  Not applicable.

\end{document}